\documentclass[reqno,12pt]{amsart}

\usepackage{amssymb}
\usepackage{verbatim}

\theoremstyle{plain}
\newtheorem{lemma}{Lemma}[section]
\newtheorem{theorem}[lemma]{Theorem}
\newtheorem{proposition}[lemma]{Proposition}
\newtheorem{corollary}[lemma]{Corollary}

\theoremstyle{example}
\newtheorem*{example}{Example}

\numberwithin{equation}{section}

\newcommand{\est}{\varnothing}

\newcommand{\seq}{\subseteq}
\newcommand{\snqq}{\subsetneqq}
\newcommand{\stm}{\setminus}

\newcommand{\longc}{,\dotsc,}
\newcommand{\longe}{=\dotsb=}
\newcommand{\longp}{+\dotsb+}
\newcommand{\longd}{\mid\dots\mid}
\newcommand{\longle}{\le\dotsb\le}
\newcommand{\longop}{\oplus\dotsb\oplus}

\newcommand{\mmod}[1]{\!\!\pmod{#1}}

\newcommand{\<}{\langle}
\renewcommand{\>}{\rangle}
\newcommand{\lpr}{\left(}
\newcommand{\rpr}{\right)}
\newcommand{\lfr}{\left\{}
\newcommand{\rfr}{\right\}}
\newcommand{\lfl}{\left\lfloor}
\newcommand{\rfl}{\right\rfloor}

\DeclareMathOperator{\diam}{diam}
\DeclareMathOperator{\ord}{ord}
\DeclareMathOperator{\rk}{rk}

\newcommand{\alp}{\alpha}
\newcommand{\bet}{\beta}
\newcommand{\eps}{\varepsilon}
\newcommand{\lam}{\lambda}
\renewcommand{\phi}{\varphi}

\newcommand{\oA}{\overline A}
\newcommand{\oG}{\overline G}

\newcommand{\oa}{\bar a}

\newcommand{\N}{{\mathbb N}}
\newcommand{\Z}{{\mathbb Z}}

\renewcommand{\sp}[1][\rho]{{\mathbf s}_{#1}^+}
\newcommand{\tp}[1][\rho]{{\mathbf t}_{#1}^+}

\begin{document}
\baselineskip=16pt

\title[Generating groups by addition]
  {Generating abelian groups by addition only}


\author{Benjamin Klopsch}
\address{Mathematisches Institut \\
  Heinrich-Heine-Universit\"at \\ 40225 D\"usseldorf \\ Germany}
\email{klopsch@math.uni-duesseldorf.de}
\author{Vsevolod F. Lev}
\address{Department of Mathematics \\ The University of Haifa at Oranim \\
  Tivon 36006 \\ Israel}
\email{seva@math.haifa.ac.il}

\thanks{An essential part of our work was carried out while the first
  named author paid a one-week visit to the University of Haifa in May
  2006. He is grateful for the hospitality experienced during his
  visit, and he acknowledges the financial support provided by the
  Deutsche Forschungsgemeinschaft.}

\begin{abstract}
We define the \emph{positive diameter} of a finite group $G$ with respect to
a generating set $A\seq G$ to be the smallest non-negative integer $n$ such
that every element of $G$ can be written as a product of at most $n$ elements
of $A$. This invariant, which we denote by $\diam_A^+(G)$, can be interpreted
as the diameter of the Cayley digraph induced by $A$ on $G$.

In this paper we study the positive diameters of a finite \emph{abelian}
group $G$ with respect to its various generating sets $A$. More specifically,
we determine the maximum possible value of $\diam_A^+(G)$ and classify all
generating sets for which this maximum value is attained. Also, we determine
the maximum possible cardinality of $A$ subject to the condition that
$\diam_A^+(G)$ is ``not too small''.

Conceptually, the problems studied are closely related to our earlier work
\cite{b:kl} and the results obtained shed a new light on the subject. Our
original motivation came from connections with caps, sum-free sets, and
quasi-perfect codes.
\end{abstract}

\maketitle


\section{Introduction}\label{s:intro}

Let $G$ be a finite group and let $A$ be a subset of $G$. Recall, that
the \emph{subgroup generated by $A$ in $G$} is the smallest subgroup of
$G$, containing $A$; in other words, it is the intersection of all
subgroups containing $A$:
  $$ \textstyle \<A\> := \bigcap\,\{H \le G \colon A \seq H\}. $$
If $\<A\>=G$, then $A$ is called a \emph{generating set for $G$}. Loosely
speaking, the aim of this paper is to determine whether and how quickly $A$
generates $G$, using only information about the cardinality of $A$.

Since eventually we restrict to the situation where $G$ is abelian, we use
additive notation throughout. Let $A_0:=A\cup\{0\}$, and for every
non-negative integer $\rho$ define
$$
  \<A\>_\rho^+ := \rho A_0 = \{a_1\longp a_\rho\colon a_i\in A_0\} \seq G,
$$
the set of all $g\in G$ representable as a sum of \emph{at most $\rho$}
elements of $A$. Then $\<A_0\>_\rho^+=\<A\>_\rho^+$ for all
$\rho\in\N_0:=\N\cup\{0\}$, and, evidently, these sets form an ascending
chain $\{0\}=\<A\>_0^+\seq\<A\>_1^+\seq\dots$ which stabilizes after finitely
many steps in the group $\<A\>=\bigcup_{\rho\ge 0}\<A\>_\rho^+$.

For every $g\in G$ we define the \emph{\textup{(}positive\textup{)}
  length of $g$ with respect to $A$} as
$$
  l_A^+(g) := \min \{ \rho\in\N_0 \colon g\in \<A\>_\rho^+ \}.
  $$
  Here we agree that $\min\est=\infty$, so that for $g\in G$ we
  have $g\in\<A\>$ if and only if $l_A^+(g)<\infty$. Observe that
  $l_A^+(g)$ is the minimum number of elements of $A$ required to
  represent $g$ as their sum. Thus, for every $\rho \in \N_0$ we have
$$
\<A\>_\rho^+ = \{ g\in G \colon l_A^+(g)\le \rho \}.
$$
We define the \emph{\textup{(}positive\textup{)} diameter of $G$
  with respect to $A$} as
\begin{align*}
  \diam_A^+(G)
    &:= \min \{ \rho\in\N_0 \colon \<A\>_\rho^+ = G \} \\
    &\;= \max \{ l_A^+(g) \colon g\in G\}.
\end{align*}
Note that the notion of a diameter allows a simple graph-theoretic
interpretation. Indeed, for any $x,y \in G$, the distance from $x$ to $y$ in
the Cayley digraph induced by $A$ on $G$ is given by $l_A^+(-x+y)$, and thus
$\diam_A^+(G)$ is the diameter of this digraph. In this sense, $\diam_A^+(G)$
gives a quantitative measure for the ``number of steps'' needed to generate
$G$ from $A$. In particular, $A$ is a generating set for $G$ if and only if
$\diam_A^+(G) < \infty$.

The length of a given element and the diameter of $G$ depend upon the
choice of $A\seq G$. In contrast, the \emph{absolute
  \textup{(}positive\textup{)} diameter}
$$
\diam^+(G) := \max \{ \diam_A^+(G) \colon \<A\>=G \}
$$
is an invariant of $G$ itself; every generating set produces the group in at
most that many steps.

The two fundamental problems addressed in this paper are:
\begin{itemize}
\item[\textbf{P1.}] What is the value of $\diam^+(G)$ and what are the
  generating sets $A \seq G$ such that $\diam_A^+(G) = \diam^+(G)$?
\item[\textbf{P2.}] How large can $\diam_A^+(G)$ be, given that $|A|$ is
  large?
\end{itemize}
Under the assumption that $G$ is \emph{abelian}, we solve the first
problem completely and we provide a partial answer to the second one. Our
main results are collected in Section~\ref{s:summary}; the proofs are
presented in Sections~\ref{s:Kneser} to~\ref{s:bounds_IV}.


We remark that different notions of length and diameter result from
considering all ``algebraic sums'' $\pm a_1 \pm \dotsb \pm a_n,\ a_i \in A$,
at each step of generating the group $G$ from its subset $A$. This approach,
leading to a significantly different theory, was pursued in~\cite{b:kl}. In
the important ``classical'' case, where $G$ is of exponent $2$, the two
approaches are identical; connections with caps, sum-free sets, and
quasi-perfect codes for this special case are discussed in
\cite[Section~3]{b:kl}.

In fact, the problem of finding generating subsets $A$ of a given finite
group $G$ such that both $|A|$ and $\diam^+_A(G)$ are small has attracted
much attention. Traditionally, most authors restrict their study to symmetric
generating sets, i.e.\ generating sets which are closed under taking
inverses. Asymptotic estimates of the diameters of non-abelian simple groups
with respect to various types of (symmetric) generating sets can be found in
the survey \cite{b:bhkls}, which also lists related work, e.g.\ on the
diameters of permutation groups. Investigations related to the diameters of
abelian (more generally nilpotent) groups and driven by the search for
efficient communication networks can be found, for instance, in \cite{b:ab}
and \cite{b:gp}.

As far as we know, this paper and its precursor~\cite{b:kl} form the first
steps towards a systematic study of the ``worst'' case, determining in
particular the largest possible value of $\diam^+_A(G)$ for all, respectively
all symmetric, generating subsets $A \seq G$. Moreover, the present paper is
one of the very few, concentrating on asymmetric generating sets --- which,
we believe, can be of interest in the theory of network design.


\section{Summary of Results}\label{s:summary}

Throughout this section let $G$ denote a finite abelian group.

\subsection{The Absolute Diameter}

Up to isomorphism, every finite abelian group is characterized by its type.
Suppose that $G$ has type $(m_1,\ldots,m_r)$; that is, $G \cong
\Z_{m_1}\longop\Z_{m_r}$ where $1\neq m_1\longd m_r$. The number $r$ is
called the \emph{rank of $G$} and denoted $\rk(G)$.  A \emph{standard
generating set for $G$} is an ordered subset $A=\{a_1\longc a_r\}\seq G$ such
that $G=\<a_1\>\longop\<a_r\>$ and $\ord(a_i)=m_i$ for all $i\in [1,r]$.

If $A$ is a standard generating set for $G$, then plainly
$\diam_A^+(G)=\sum_{i=1}^r (m_i-1)$. Perhaps not surprisingly, this is as
large as possible: specifically, we have
\begin{theorem}\label{t:thm21}
  A finite abelian group $G$ of type $(m_1,\ldots,m_r)$ has diameter
  $$
  \diam^+(G) = \sum_{i=1}^r (m_i-1).
  $$
\end{theorem}

Conversely, if $A$ is a generating set for $G$ with
$\diam_A^+(G)=\diam^+(G)$, then $A$ is ``nearly standard''.

\begin{theorem}\label{t:thm22}
Let $G$ be a finite abelian group of rank $r$. Then  for every
subset $A \seq G$ the following assertions are equivalent:
\begin{itemize}
\item[(i)]
  $\diam_A^+(G) = \diam^+(G)$;
\item[(ii)] There exists a standard generating set $\{e_1 \longc e_r\}$
  for $G$ with the following property: one can write
  $A_0=\{0,a_1,\ldots,a_r\}$ so that for each $i\in[1,r]$, either
  $a_i=e_i$, or $a_i=e_i+a_{\sigma(i)}$ with $\sigma(i)\in[i+1,r]$.
\end{itemize}
Moreover, if $G$ has type $(m_1\longc m_r)$ and if $A$ is as in {\rm
  (ii)}, then every element $g\in G$ can be uniquely written as
$g=\sum_{i=1}^r \lam_i a_i$ with $\lam_i\in[0,m_i-1]$ for all
$i\in[1,r]$, and then $l_A^+(g)=\sum_{i=1}^r \lam_i$.
\end{theorem}

The assertion (ii) above describes a general procedure which allows one to
construct any set $A=\{a_1\longc a_r \}$ with $\diam_A^+(G)=\diam^+(G)$ from
a standard generating set $\{e_1\longc e_r\}$, starting from $a_r=e_r$ and
working the way down to $a_1$.

\subsection{The Maximum Size of Large Diameter Sets}
For abelian groups, Theorems~\ref{t:thm21} and \ref{t:thm22} provide a
complete solution of Problem~P1. Now we address Problem~P2; that is, we want
to determine to what extent the size of a generating set guarantees fast
generation.

Our aim is to find, for $\rho\in\N$, an upper bound for the sizes of
generating sets $A$ for $G$ with $\diam_A^+(G)\ge\rho$. Agreeing that
$\max\est=0$, we define
\begin{align*}
  \sp(G) &:=   \max \{ |A| \colon A\seq G \text{ such that } \rho
                                          \le \diam_A^+(G) < \infty \} \\
          &\;= \max \{ |A| \colon A\seq G \text{ such that }
                                       \<A\>_{\rho-1}^+ \ne \<A\>=G \}.
\end{align*}
The significance of this invariant stems from the observation that every
generating set $A$ for $G$ of size, larger than $\sp(G)$, surely
generates $G$ in less than $\rho$ steps. It is immediate that
$$
\sp[1](G) \geq \sp[2](G) \geq \dotsb \quad \text{and} \quad \sp(G) = 0
\text{ for $\rho > \diam^+(G)$.}
$$

In order to better understand $\sp(G)$ we introduce a related and perhaps
more fundamental invariant, $\tp(G)$; this requires a short preparation. We
say that a subset $A\seq G$ is \emph{$\rho$-maximal} if it is maximal (under
inclusion) subject to $\diam_A^+(G)\ge\rho$; that is, subject to
$\<A\>_{\rho-1}^+\neq G$. Plainly, we have
  $$ \sp(G) = \max \{ |A| \colon \text{$A$ is a $\rho$-maximal
                                           generating set for $G$}\}. $$
Recall that the \emph{period} of a subset $S\seq G$ is the subgroup
$\pi(S):=\{g\in G\colon S+g=S\}\le G$, and $S$ is \emph{periodic} or
\emph{aperiodic} according to whether $\pi(S)\neq\{0\}$ or $\pi(S)=\{0\}$.

It is not difficult to see that $\rho$-maximal generating sets for a group
$G$ can be constructed by a ``lifting process'' from aperiodic $\rho$-maximal
generating sets for the quotient groups of $G$.%
\footnote{Suppose that $H \lneqq G$, and let $\phi\colon G\rightarrow G/H$
  denote the canonical homomorphism. If $\oA$ is a generating set for $G/H$
  with $0\in\oA$, then the full pre-image $A:=\phi^{-1}(\oA)$ is a generating
  set for $G$ with $\diam_A^+(G)=\diam_{\bar A}^+(G/H)$ and $H\le\pi(A)$.
  Conversely, suppose that $A$ is a generating set for $G$ with $0\in A$ and
  $H \le \pi(A)$. Then the image $\oA=\phi(A)$ is a generating set for $G/H$
  and $\diam_{\bar A}^+(G/H)=\diam_A^+(G)$.  Furthermore, if $A$ is
  $\rho$-maximal then so is $\oA$, and if $H=\pi(A)$ then $\oA$ is
  aperiodic.}
This shows that every $\rho$-maximal generating set is induced by an
aperiodic one, and a natural point of view, therefore, is to consider
aperiodic $\rho$-maximal generating sets as ``primitive'' and concentrate on
their properties first. Accordingly, we let
$$
\tp(G) := \max \{ |A| \colon \text{$A$ is an aperiodic
  $\rho$-maximal generating set for $G$} \}.
$$

A slightly simpler description of $\tp(G)$ can be provided. Suppose that
$A\seq G$ is aperiodic and $\rho$-maximal, but $\<A\>\neq G$. Then, being
$\rho$-maximal, $A$ coincides with a maximal subgroup of $G$. Since $A$ is
aperiodic, we have $A=\{0\}$, and $G$ is cyclic of prime order. Thus any
non-empty subset $B \subseteq G$, other than $\{0\}$, satisfies
$\<B\>^+_{\rho-1}=G$. This implies $\diam^+(G)<\rho$, and consequently
$\tp(G)=0$. Therefore,
$$
\tp(G) =
\begin{cases}
  \max \{ |A| \colon \text{$A\seq G$ is aperiodic and $\rho$-maximal}
  \} & \text{if $\rho \in [1,\diam^+(G)]$,} \\
  0 & \text{otherwise.}
\end{cases}
$$

The close connection between $\tp(G)$ and $\sp(G)$ is described by
\begin{lemma}\label{l:lem23}
Let $G$ be a finite abelian group and let $\rho\ge 2$. Then
  $$ \sp(G) = \max \{ |H|\cdot\tp(G/H) \colon H\lneqq G \}. $$
\end{lemma}

\subsubsection{Explicit Formulae for $\tp(G)$ and $\sp(G)$}

Suppose that $\rho\in[1,\diam^+(G)]$. We determine $\tp(G)$ and
$\sp(G)$ explicitly in the following cases:
\begin{itemize}
\item[(i)] $G$ is an arbitrary finite abelian group and
  $\rho\in\{1,2,3,4,\diam^+(G)\}$, with some exceptions regarding
  $\tp[4](G)$;
\item[(ii)] $G$ is a finite cyclic group and $\rho\in[1,\diam^+(G)]$ is
  arbitrary.
\end{itemize}

The case of $G$ arbitrary and $\rho\in\{1,2,3,\diam^+(G)\}$ is settled by
\begin{theorem}\label{t:thm24}
Let $G$ be a finite abelian group. We have
\begin{itemize}
\item[(i)]   $\tp[1](G)=0$ and $\sp[1](G)=|G|$, provided that
             $\diam^+(G)\ge 1$;
\item[(ii)]  $\tp[2](G)=\sp[2](G)=|G|-1$, provided that $\diam^+(G)\ge 2$;
\item[(iii)] $\tp[3](G)=\sp[3](G)=\lfl |G|/2 \rfl$, provided that
             $\diam^+(G)\ge 3$;
\item[(iv)]  if $\rho=\diam^+(G) \ge 2$, then $\tp(G)=\sp(G)=\rk(G)+1$.
\end{itemize}
\end{theorem}

Recall that by Theorem~\ref{t:thm21}, for every $m\in\N$ the diameter of the
cyclic group $\Z_m$ of order $m$ is $\diam^+(\Z_m)=m-1$. The precise values
of $\tp(\Z_m)$ and $\sp(\Z_m)$ are established by
\begin{theorem}\label{t:thm25}
  Let $m\in\N$ and $\rho\in [2,m-1]$. Then
\begin{align*}
  \tp(\Z_m) &= \lfl \frac{m-2}{\rho-1} \rfl +1, \\
  \sp(\Z_m) &= \max \lfr \frac md \lpr \lfl \frac{d-2}{\rho-1} \rfl
                            + 1 \rpr \colon d\mid m,\ d\ge \rho+1 \rfr.
\end{align*}
\end{theorem}

\begin{corollary}
Let $p$ be a prime and let $\rho\in[2,p-1]$. Then
$$ \tp(\Z_p) = \sp(\Z_p) = \lfl \frac{p-2}{\rho-1} \rfl + 1. $$
\end{corollary}

The case of $G$ arbitrary and $\rho=4$ remains partially open and
deserves further comments. A remarkable result of Davydov and Tombak
\cite{b:dt} motivated by coding theory and finite geometry
applications (see \cite{b:chll,b:fp} and also \cite{b:l} for some
extensions) shows that
\begin{equation}\label{e:equ21}
  \tp[4](\Z_2^r) = 2^{r-2}+1 \quad\text{and}\quad
                        \sp[4](\Z_2^r) = 5\cdot2^{r-4};\qquad r\ge 4.
\end{equation}
We notice that the first formula is considerably subtler than the second.
Indeed, applying Lemma~\ref{l:lem23} one can easily compute $\sp[4](\Z_2^r)$
from the values $\tp[4](\Z_2^s)$, $s\le r$.

We obtain similarly explicit formulae for abelian groups, which are
not of exponent $2$, except that we have been unable to determine the
exact value of $\tp[4](G)$ for some particularly tough groups $G$. To
state our results we introduce the new invariant
$$
  \eta(G) := \max \{ |H| \colon H \le G \text{ with } \exp(G/H) \neq
                        2 \text{ and } |G/H| \equiv 2\mmod 3 \} \in \N_0.
$$

\begin{theorem} \label{t:thm27}
Let $G$ be a finite abelian group with $\diam^+(G) \ge 4$.
\begin{itemize}
\item[(i)] We have
  $$
  \tp[4](G) =
  \begin{cases}
    \frac13\,|G|     & \text{if $3$ divides $|G|$,} \\
    \frac13\,(|G|-1) & \text{if every divisor of $|G|$ is congruent
      to $1$ modulo $3$.}
  \end{cases}
  $$
\item[(ii)] If $G$ is not of exponent $2$, then
  $$
  \sp[4](G) =
  \begin{cases}
    \frac{1}{3}\, (|G| + \eta(G))  & \text{if $\eta(G)\neq 0$,} \\
    \frac{1}{3}\, |G| & \text{if $\eta(G) = 0$ and $3$ divides $|G|$,} \\
    \frac{1}{3}\, (|G|-1) & \text{if $\eta(G)=0$ and $3$ does not
      divide $|G|$.}
  \end{cases}
  $$
\end{itemize}
\end{theorem}

It is an interesting open problem to determine $\tp[4](G)$ for groups $G$,
not covered by Theorem~\ref{t:thm27} (i) and not of exponent $2$. The first
case one may wish to consider is $G \cong \Z_5^r$, $r \in \N$. Indeed, we
have made some progress towards proving that $\tp[4](\Z_5^r)= (3\cdot
5^{r-1}-1)/2$ for $r\ge 2$. It is our intention to give a comprehensive
treatment in a separate paper.

\subsubsection{General Estimates for $\tp(G)$ and $\sp(G)$}

Suppose that $\rho\ge 5$. We expect that finding general explicit
formulae for $\tp(G)$ and $\sp(G)$ for finite abelian groups $G$ with
$\diam^+(G)\ge 5$ is very difficult, if at all feasible.  Indeed, it
is likely that for increasing values of $\rho$, the dependence of
$\tp(G)$ and $\sp(G)$ on the algebraic structure of $G$ becomes overly
complicated. Nevertheless, some general estimates in terms of $\rho$
and $|G|$ can be given; to some extent, these estimates compensate for
the lack of explicit formulae as those we were able to provide for
$\rho\le 4$.

\begin{proposition} \label{p:prop28}
Let $G$ be a finite abelian group, and suppose that $\rho\ge 2$. Then
  $$ \tp(G) \le \lfl \frac{|G|-2}{\rho-1} \rfl + 1. $$
\end{proposition}

The next theorem determines the maximum possible cardinality of a
generating set $A$ with $\diam_A^+(G)\ge\rho$ and describes the structure
of all sets $A$ of this cardinality.
\begin{theorem} \label{t:thm29}
Let $G$ be a finite abelian group, and suppose that $\rho\ge 4$. Then
$\sp(G)\le\frac2{\rho+1}\,|G|$. Moreover, equality holds if and only if
$G$ has a subgroup $H$ such that $G/H$ is cyclic of order $\rho+1$;
indeed, the following assertions are equivalent:
\begin{itemize}
\item[(i)] $A$ is a generating set for $G$ satisfying $\diam_A^+(G)\ge\rho$
  and $|A|=\frac2{\rho+1}\,|G|$;
\item[(ii)] $A=H\cup(g+H)$, where $H$ is a subgroup of $G$ such that $G/H$ is
  cyclic of order $\rho+1$, and $g$ is an element of $G$ such that $g+H$ is a
  generator of $G/H$.
\end{itemize}
\end{theorem}

Our final result goes somewhat beyond Theorem~\ref{t:thm29}, establishing the
structure of those generating sets $A$ with $\diam_A^+(G)\ge\rho$ and
$|A|>\frac32\,\rho^{-1}|G|$.

\begin{theorem} \label{t:thm210}
Let $G$ be a finite abelian group, and suppose that $\rho\ge 4$. If $A$ is a
generating set for $G$ such that $\diam_A^+(G)\ge\rho$ and
$|A|>\frac32\,\rho^{-1}|G|$, then there exist $H\le G$ and $g\in G$
satisfying the following conditions:
\begin{itemize}
\item[(i)]  $A\seq H\cup(g+H)$, and $|A|>\big(2-\frac{\rho-3}{2\rho}\big)|H|$;
\item[(ii)] $G/H$ is cyclic, its order satisfies
  $\rho+1\le |G/H|< \frac43\,\rho$, and $g+H$ generates $G/H$.
\end{itemize}
\end{theorem}


\subsection{Organization and List of Notation}

An important tool used in our arguments is Kneser's Theorem, stated in
Section~\ref{s:Kneser}. The proofs of our results, presented above, are
collected in Sections~\ref{s:abs_diam} to \ref{s:bounds_IV}. In the Appendix
we discuss an open problem, related to the quantity $\tp(G)$.

For convenience we include a list of standard notation, used throughout the
paper.
\begin{tabbing}
  {gcd(n 1, . . .,n r) } \= \kill
  \, $\N,\N_0$, and $\Z$ \> the sets of positive, non-negative, and
  all integers, respectively.
\end{tabbing}
For natural numbers $m,r,n_1, \ldots, n_r$:
\begin{tabbing}
  {gcd(n 1, . . .,n r) } \= \kill
  \, $\Z_m$  \> the group $\Z / m\Z$ of residues modulo $m$; \\
  \, $\gcd(n_1\longc n_r)$ \> the greatest common divisor of $n_1\longc
   n_r$.
\end{tabbing}
For real numbers $x,y$:
\begin{tabbing}
  {gcd(n 1, . . .,n r) } \= \kill
  \, $\lfloor x\rfloor$ \> the greatest integer not exceeding $x$; \\
  \, $[x,y]$ \> the set of all integers $n$ satisfying $x \le n
   \le y$.
\end{tabbing}
For a group $G$, an element $g\in G$, and a subset $A\subseteq G$:
\begin{tabbing}
  {gcd(n 1, . . .,n r) } \= \kill
  \, $\rk(G)$ \> the rank (i.e.\ the minimal number of generators) of $G$; \\
  \, $\ord(g)$ \> the order of $g$; \\
  \, $\langle A\rangle$ \> the subgroup of $G$, generated by $A$; \\
  \, $\pi(A)$ \> the period $\{ g \in G \colon A + g = A \}$ of $A$ in $G$.
\end{tabbing}
For a non-negative integer $r$ and subsets $A,A_1 \longc A_r$ of an abelian
group $G$:
\begin{tabbing}
  {gcd(n 1, . . .,n r) } \= \kill
  \, $A_1 \longp A_r$ \> the set $\{a_1 \longp a_r \colon a_1 \in A_1
  \longc a_r \in A_r \}$; \\
  \, $rA$ \> the set $A\longp A$ ($r$ summands); \\
  \, $r \ast A$ \> the set $\{ra \colon a \in A \}$.
\end{tabbing}
Also, for integers $m$ and $n$ we write $m \mid n$ to indicate that $m$
divides $n$.


\section{Kneser's Theorem} \label{s:Kneser}

We recall here a fundamental result due to Kneser, used in the proofs of
Proposition~\ref{p:prop28} and Theorems~\ref{t:thm29} and~\ref{t:thm210}.

\begin{theorem}[Kneser \cite{b:kn1,b:kn2}; see also \cite{b:mann}]
Let $A$ and $B$ be finite, non-empty subsets of an abelian group $G$ such
that
  $$ |A+B| \le |A|+|B|-1. $$
Then, letting $H := \pi(A+B)$, we have
  $$ |A+B| = |A+H|+|B+H|-|H|. $$
\end{theorem}

Since, in the above notation, we have $|A+H|\ge |A|$ and $|B+H|\ge |B|$,
Kneser's Theorem shows that $|A+B|\ge |A|+|B|-|H|$. A straightforward
induction yields

\begin{corollary}
Let $A_1\longc A_r$ be finite, non-empty subsets of an abelian group $G$,
and write $H:=\pi(A_1\longp A_r)$. Then
  $$ |A_1\longp A_r| \ge |A_1|\longp |A_r|-(r-1)|H|. $$
In particular, if $A_1 \longp A_r$ is aperiodic, then
  $$ |A_1\longp A_r|\ge|A_1|\longp|A_r|-(r-1). $$
\end{corollary}


\section{The absolute diameter: proofs of Theorems~\ref{t:thm21}
  and~\ref{t:thm22}}\label{s:abs_diam}

\begin{proof}[Proof of Theorem~\ref{t:thm21}]
Suppose that $G \cong \Z_{m_1} \longop \Z_{m_r}$ where $1\neq m_1\longd
m_r$. If $A$ is a standard generating set for $G$, it is easily seen that
$\diam_A^+(G)=\sum_{i=1}^r (m_i -1)$. It remains to show that
 $\diam^+(G)\le \sum_{i=1}^r (m_i -1)$.

We argue by induction on $|G|$. For $|G|=1$ there is nothing to prove.
Next, consider the special case where $G$ is of prime exponent. Then
every generating set for $G$ necessarily contains a standard generating
set, namely a basis for $G$ regarded as a vector space, and the claim
follows.

Now return to the general case $|G|>1$. Let $A$ be a generating set for $G$,
and choose a prime $p$ dividing $m_1$. Clearly, $p\ast A=\{pa\colon a\in A\}$
is a generating set for the subgroup $p\ast G=\{pg\colon g\in G\}$ of $G$.
Write ${\oG}:=G/(p\ast G)$, so that $\oG$ is elementary abelian of exponent
$p$ and rank $r$, and let $\oA$ be the image of $A$ under the natural
homomorphism $G\to\oG$. Then the special case, considered above, and
induction applied to the group $p\ast G$ yield
\begin{align*}
  \diam_A^+(G) & \le \diam_{\bar A}^+({\oG}) + p \cdot
  \diam_{p \ast A}^+(p \ast G) \\
  & \le r(p-1) + p \sum_{i = 1}^r (m_i/p -1) \\
  & = \sum_{i =1}^r (m_i -1).
\end{align*}
\nopagebreak
\end{proof}

\begin{proof}[Proof of Theorem~\ref{t:thm22}]
Suppose that $G \cong \Z_{m_1} \longop \Z_{m_r}$ where $1 \neq m_1 \longd
m_r$, and let $A \subseteq G$. Without loss of generality we assume that $0
\in A$, that is $A=A_0$.

First we verify the last assertion of the theorem; in view of
Theorem~\ref{t:thm21}, this also yields the implication from (ii) to (i).
Suppose that $\{e_1 \longc e_r \}$ is a standard generating set for $G$ with
the property that $A=\{ 0, a_1 \longc a_r \}$, where for each $i \in [1,r]$,
either $a_i = e_i$, or $a_i = e_i + a_{\sigma(i)}$ with $\sigma(i) \in
[i+1,r]$. We show that for every $g \in G$ there exist $\lambda_1 \longc
\lambda_r \in \N_0$ with $g = \sum_{i = 1}^r \lambda_i a_i$, $l_A^+(g) =
\sum_{i = 1}^r \lambda_i$, and $\lambda_i \le m_i -1$ for all $i \in [1,r]$.
As $|G| = \prod_{i=1}^r m_i$, this also implies the uniqueness. Fix $g \in
G$, and choose $\mu_1 \longc \mu_r \in \N_0$ such that $g = \sum_{i=1}^r
\mu_i a_i$ and $l_A^+(g) = \sum_{i=1}^r \mu_i$. If the set $J := \{ j \in
[1,r] \colon \mu_j \ge m_j \}$ is empty, there is nothing further to prove.
Assume that $J\neq \varnothing$ and let $j := \min J$. As $\mu_j e_j = (\mu_j
- m_j) e_j$ and $l_A^+(g) = \sum_{i=1}^r \mu_i$, we must have $a_j = e_j +
a_{\sigma(j)}$ with $\sigma(j) \in [j+1,r]$. But then
$$ \mu_j a_j = (\mu_j - m_j) a_j + m_j a_j
                                   = (\mu_j - m_j) a_j + m_j a_{\sigma(j)}, $$
and replacing $\mu_i$ by
$$
\mu_i' :=
\begin{cases}
\mu_i & \text{if $i \not \in \{j,\sigma(j)\}$,} \\
\mu_j - m_j & \text{if $i = j$,} \\
\mu_{\sigma(j)} + m_j & \text{if $i = \sigma(j)$,}
\end{cases}
$$
we arrive at $g = \sum_{i=1}^r \mu_i' a_i$ with $l_A^+(g) = \sum_{i=1}^r
\mu_i'$ and $\mu_i' \ge 0$ for $i \in [1,r]$. Moreover, $\mu_i' = \mu_i$
for $i \in [1,j-1]$ and $\mu_j' < \mu_j$. Iteration of this process
finitely many times yields $\lambda_1 \longc \lambda_r \in \N_0$ such
that $g = \sum_{i = 1}^r \lambda_i a_i$, $l_A^+(g) = \sum_{i = 1}^r
\lambda_i$, and $\lambda_i \le m_i -1$ for all $i \in [1,r]$, as wanted.

We now assume (i), i.e.\ $\diam_A^+(G) = \diam^+(G)$, and derive (ii),
using induction on $|G|$. Clearly, there is nothing to prove if $|G|=1$,
and we assume that $|G|>1$.

Our first claim is that $A$ contains an element of order $m_r=\exp(G)$.
If $m_r$ is prime, then \emph{all} non-zero elements of $G$ are of order
$m_r$. Suppose now that $m_r$ is composite. Choose a prime $p$ dividing
$m_1$, and note that $m_r/p
> 1$. Write ${\oG} := G/(p \ast G)$ and let $\oA$ be the image of $A$ under
the natural homomorphism $G\to\oG$. Now revisit the proof of
Theorem~\ref{t:thm21}: our assumption $\diam_A^+(G) = \diam^+(G)$ implies
that
$$
\diam_{p \ast A}^+(p \ast G) = \sum_{i = 1}^r (m_i/p -1) = \diam^+(p \ast
G).
$$
Choose $s\in[1,r]$ so that $m_1\longe m_{s-1}=p<m_s$. Then the group
 $p\ast G$ is of type $(m_s/p\longc m_r/p)$ and using the induction
hypothesis, applied to this group and the generating set $p\ast A$, we
write $p\ast A=\{0,b_s\longc b_r\}$ so that $\ord(b_r)=m_r/p$ and every
element of $p\ast G$ has a (unique) representation as $\sum_{i=s}^r
\lambda_i b_i$, where the coefficients satisfy $0\le \lambda_i<m_i/p$ for
all $i\in[s,r]$. We choose $a_r\in A$ so that $b_r=pa_r$ and show that
$\ord(a_r)=m_r$. Suppose, for a contradiction, that $\ord(a_r)<m_r$. Then
$\ord(a_r)$ is a proper divisor of $m_r$ while, at the same time,
$\ord(a_r)$ is divisible by $\ord(pa_r)=m_r/p$. It follows that
$\ord(a_r)=m_r/p$ and $\<a_r\>=\<pa_r\>$; therefore every element of
$p\ast G$ has a (unique) representation as $\sum_{i=s}^{r-1}\lam_i
b_i+\lam_r a_r$ with $0\le\lam_i<m_i/p$ for all $i\in [s,r]$.
Consequently, conferring to the proof of Theorem~\ref{t:thm21}, we have
\begin{align*}
  \diam_A^+(G)
    &\le \diam_{\bar A}^+({\oG})
                      + p \sum_{i=s}^{r-1} (m_i/p -1) + (m_r/p -1) \\
    &=   r(p-1) + p \sum_{i=s}^r (m_i/p -1) - (p-1)(m_r/p -1) \\
    &<   \sum_{i=1}^r (m_i -1) \\
    &=   \diam^+(G),
\end{align*}
in contradiction to (i). Thus our claim is proved: $A$ contains an
element of order $m_r$.

The next step shows among other things that $|A|=r+1$. Choose $a \in A$
with $\ord(a)=m_r$. Recalling that in a finite abelian group every cyclic
subgroup of maximal possible order admits a complement, we find a
subgroup $H\le G$ such that $G=H\oplus\langle a\rangle$, and consequently
$H\cong\Z_{m_1}\longop\Z_{m_{r-1}}$. Let $\oA$ denote the image of $A$
under the projection of $G$ onto $H$ along $\<a\>$. From
  $$ \diam_{\bar A}^+(H)+(m_r-1) \ge \diam^+_A(G)
                                      = \diam^+(G) = \diam^+(H) + (m_r-1) $$
it follows that $\diam_{\bar A}^+(H)=\diam^+(H)$. Applying the induction
hypothesis to the group $H$, we find a standard generating set
$\{h_1\longc h_{r-1}\}$ for $H$ so that, writing $\oA=\{0,\oa_1\longc
\oa_{r-1}\}$, for each $i \in [1,r-1]$ we either have $\oa_i = h_i$, or
$\oa_i = h_i + \oa_{\sigma(i)}$ with $\sigma(i) \in [i+1,r]$. For every
$i\in[1,r-1]$ fix arbitrarily $a_i\in A$ with $a_i\in\oa_i+\<a\>$, and
let $a_r:=a$.

Since every element of $H$ can be written as $\sum_{i=1}^{r-1}\lam_i\oa_i$
with $\lam_i\in[0,m_i-1]$ for all $i\in[1,r-1]$, every element of $G$ has a
representation as $\sum_{i=1}^r \lam_i a_i$ with $\lam_i\in[0,m_i-1]$ for all
$i\in[1,r]$. This shows that any $g\in G$ with
 $g\neq\sum_{i=1}^r (m_i-1) a_i$ satisfies
$l_A^+(g)<\sum_{i=1}^r (m_i-1)=\diam^+(G)$, and hence
 $l_A^+\big(\sum_{i=1}^r (m_i-1) a_i\big)=\sum_{i=1}^r (m_i-1)$ by~(i).
Consequently, for any choice of
 $\lambda_i\in [0,m_i-1],\ i\in [1,r]$, we have
\begin{equation}\label{e:equ41}
  l_A^+\bigg( \sum_{i=1}^r \lambda_i a_i \bigg) = \sum_{i=1}^r \lambda_i.
\end{equation}
In particular, it follows that $A$ contains no elements, other than $0$
and $a_1\longc a_r$.

Finally, we are ready to complete the proof of (ii). Since $a_i\in
h_i+\<a_{i+1}\longc a_r\>$ for all $i\in[1,r-1]$, the set $\tilde
A:=\{a_2\longc a_r \}$ generates a proper subgroup $K\lneqq G$ isomorphic
to $\Z_{m_2}\longop \Z_{m_r}$, which is complemented by $\< h_1 \>$ in
$G$. By \eqref{e:equ41} and Theorem~\ref{t:thm21} we have
$\diam_{\tilde A}^+(K) = \diam^+(K)$. Therefore, applying the induction
hypothesis to the group $K$ and the generating set
 $\tilde A$, we find a standard generating set $\{e_2,\longc e_r\}$ for $K$
such that, renumbering the elements of $\tilde A$, if necessary, for each
$i\in[2,r]$ we have either $a_i=e_i$, or $a_i=e_i+a_{\sigma(i)}$ with
$\sigma(i)\in[i+1,r]$.

We consider two cases. If $\ord(a_1)=m_1$, we put $e_1:=a_1$; it is then
immediate that $\{e_1,\ldots,e_r\}$ is a standard generating set for $G$,
satisfying our requirements.

Now assume that $\ord(a_1)>m_1$, and write $a_1=\mu_1 h_1+\sum_{i=2}^r
\mu_i e_i$ with $\mu_i\in[0,m_i-1]$ for all $i\in[1,r]$. As $G/K \cong
\Z_{m_1}$, we have $m_1a_1\in K$, and, applying induction to the group
$K$, we obtain two expressions for $m_1 a_1$:
\begin{equation}\label{e:equ42}
  m_1a_1=\sum_{i=2}^r m_1\mu_i e_i \ \text{ and } \ m_1a_1 = \sum_{i=2}^r \lam_i
  a_i,
\end{equation}
where $\lam_i\in[0,m_i-1]$ for all $i\in[2,r]$ and $\sum_{i=2}^r
\lam_i=l^+_{\tilde A}(m_1a_1)$.

For each $i\in[2,r+1]$ let $K_i:=\<e_i\longc e_r\> = \<a_i\longc a_r\>$, and
notice that this provides a descending chain of subgroups with $K_2=K$. Since
$m_1 a_1 \in K \setminus \{0\}$, we find $j \in [2,r]$ such that $m_1a_1 \in
K_j\stm K_{j+1}$. From $m_1a_1\in K_j$ and \eqref{e:equ42} we derive that
\begin{equation}\label{e:equ43}
  m_1 \mu_i \equiv 0 \mmod{m_i} \ \text{ and } \ \lam_i=0
                                              \ \text{for all}\ i\in[2,j-1].
\end{equation}
As $m_1a_1\notin K_{j+1}$, this further implies
\begin{equation}\label{e:equ44}
 m_1\mu_j \not \equiv 0 \mmod{m_j}.
\end{equation}
Comparing \eqref{e:equ42} and \eqref{e:equ43} we obtain $m_1\mu_je_j
-\lam_ja_j\in K_{j+1}$, whence $(m_1\mu_j-\lam_j)e_j\in K_{j+1}$. This
shows that $m_1\mu_j \equiv \lam_j \mmod{m_j}$, and therefore
$\lam_j\in[0,m_j-1]$ is the least non-negative residue of $m_1 \mu_j$
modulo $m_j$. In view of $m_1 \mid m_j$ and \eqref{e:equ44} we conclude
that $\lam_j \ge m_1$. On the other hand, it is not difficult to derive
from \eqref{e:equ41} that the length function $l_A^+$ restricted to
$K$ agrees with $l_{\tilde A}^+$, and hence
 $$ m_1 \le \lam_j \le \sum_{i=j}^r \lam_i = l_{\tilde A}^+(m_1 a_1) = l^+_A(m_1a_1) \le m_1.$$
This implies $\lam_j = m_1$ and $\lam_i = 0$ for $i \in [j+1,r]$;
consequently we have $m_1a_1 = m_1a_j$. Putting $e_1:= a_1 - a_j$, we now
obtain a standard generating set $\{e_1,\ldots,e_r\}$ for $G$ which,
clearly, satisfies the requirements.
\end{proof}


\section{Proofs of Lemma~\ref{l:lem23}, Theorems~\ref{t:thm24}
and~\ref{t:thm25}, and Proposition~\ref{p:prop28}}\label{s:bounds_I}

\begin{proof}[Proof of Lemma~\ref{l:lem23}]
We have to show that
$$
  \sp(G) = \max \{ |H| \cdot \tp(G/H) \colon H \lneqq G \}.
$$
If $\rho>\diam^+(G)$, then $\sp(G)=0$ and $\tp(G/H)=0$ for all $H\lneqq
G$; assume now that $\rho\in [2,\diam^+(G)]$.

Suppose that $A$ is a $\rho$-maximal generating set for $G$, and write
$H:=\pi(A)\lneqq G$. Then the image $\oA$ of $A$ under the canonical
homomorphism $G\to G/H$ is $\rho$-maximal in $G/H$, non-zero, and
aperiodic.  We get $\tp(G/H)\ge |\oA|=|A|/|H|$, and if $A$ is chosen so
that $|A|=\sp(G)$, this yields
$$ \sp(G) \le |H|\cdot\tp(G/H). $$

Conversely, let $H\lneqq G$ and suppose that $\oA$ is an aperiodic
$\rho$-maximal generating set for $G/H$. Let $A$ denote the full
pre-image of $\oA$ in $G$ under the canonical homomorphism $G\to G/H$.
Then $A$ is a generating set for $G$ with $\diam_A^+(G)\ge\rho$. If $\oA$
is chosen so that $|\oA|=\tp(G/H)$, we get
$$
\sp(G) \ge |A| = |H| \cdot |\oA| = |H| \cdot \tp(G/H).
$$
\end{proof}

\begin{proof}[Proof of Theorem~\ref{t:thm24}]
(i) If $\diam^+(G) \ge 1$, then $G$ is not trivial and the only $1$-maximal
subset of $G$ is $G$ itself; so $\tp[1](G) = 0$ and $\sp[1](G)= |G|$.

\smallskip\noindent
(ii) Suppose that $\diam^+(G)\ge 2$, or equivalently $|G|>2$. Fix a non-zero
element $g\in G$ and set $A:=G\stm\{g\}$. Then $A$ is aperiodic, generating,
and $\<A\>_1^+=A_0=A \neq G$. Thus $\tp[2](G)=\sp[2](G)=|G|-1$.

\smallskip\noindent
(iii) Suppose that $\diam^+(G)\ge 3$, which by Theorem~\ref{t:thm21} is
equivalent to $G\cong\Z_4$ or $|G|\ge 5$. If $A\seq G$ satisfies $|A|>|G|/2$,
then $\<A\>_2^+\supseteq 2A=G$ by the box principle. Thus $\sp[3](G)\le \lfl
|G|/2 \rfl$, and to complete the proof it suffices to construct an aperiodic
generating subset $A \seq G$ of cardinality $|A|=\lfl |G|/2 \rfl$ such that
$\<A\>_2^+\neq G$.

Assume first that $G$ has even order. Fix a subgroup $H\le G$ with
$|G/H|=2$ and choose an element $g\in G\stm H$. Set
$A:=\{0\}\cup(g+H)\stm\{g\}$. Clearly, we have $\<A\>=G$ and $|A|=|G|/2$.
Moreover it is easily seen that $\<A\>_2^+=2A=G\stm\{g\}$. In particular,
$2A$ is aperiodic, and so is $A$.

Now assume that $G$ has odd order. Fix $g\in G\stm\{0\}$, and let
 $h\in G$ be the (unique) element satisfying $2h=g$. The set $G\stm\{h\}$
can be (uniquely) partitioned into $(|G|-1)/2$ pairs so that the two
elements of each pair add up to $g$. Form a subset $A\seq G$ by selecting
one element out of each pair, subject to the requirement that $0$ is
selected (and therefore $g$ is not). The construction ensures that
$|A|=(|G|-1)/2$ and $g\notin 2A=\<A\>_2^+$.  Note that $\<A\>=G$, because
proper subgroups of $G$ have cardinality at most $|G|/3<|A|$. Furthermore
$A$ is aperiodic in view of $\gcd(|A|,|G|)=1$.

\smallskip\noindent
(iv) Suppose that $\rho=\diam^+(G)\ge 2$ and write $r:=\rk(G)$.
Theorem~\ref{t:thm22} shows that $\sp(G)=r+1$ and, moreover, that, if $A$ is
a standard generating set, then $A_0$ is $\rho$-maximal. It remains to
observe that the set $A_0$ is aperiodic, as one can easily verify taking into
account that $|G|>2$.
\end{proof}

The following lemma prepares the proof of Proposition~\ref{p:prop28}
which, in turn, is used to prove Theorem~\ref{t:thm25}.

\begin{lemma}\label{lem51}
Let $A$ be a $\rho$-maximal subset of a finite abelian group $G$, where
$\rho\in[2,\diam^+(G)]$. Then $\pi(A)=\pi(\<A\>_\tau^+)$ for all
$\tau\in[1,\rho-1]$.
\end{lemma}

\begin{proof}
Write $H:=\pi(\<A\>_{\rho-1}^+)$. Since $A$ is $\rho$-maximal, we have
$A_0=A$; consequently, $\<A\>_\tau^+ = \tau A$ holds for every
$\tau\in[1,\rho-1]$, and hence
$\pi(A)\le\pi(\<A\>_2^+)\longle\pi(\<A\>_{\rho-1}^+)=H$. On the other
hand, since
  $$ \<A+H\>_{\rho-1}^+ = \<A\>_{\rho-1}^++H = \<A\>_{\rho-1}^+ \snqq G, $$
by maximality of $A$ we have $A+H=A$, whence $H\le\pi(A)$.
\end{proof}

\begin{proof}[Proof of Proposition~\ref{p:prop28}]
Let $A$ be an aperiodic $\rho$-maximal subset of $G$. By Lemma~\ref{lem51}
we have $\pi(\<A\>_\tau^+)=\{0\}$ for all $\tau\in[1,\rho-1]$. Now Kneser's
Theorem shows that
$$
  |G|-1 \ge |\<A\>_{\rho-1}^+| = |(\rho-1)A_0| \ge
                                            (\rho-1)|A|-(\rho-2),
$$
and hence $|A|\le (|G|-2)/(\rho-1)+1$. The assertion follows.
\end{proof}

\begin{proof}[Proof of Theorem~\ref{t:thm25}]
It suffices to show that
  $$ \tp(\Z_m) = \lfl \frac{m-2}{\rho-1} \rfl + 1 $$
for any given $m\in\N$ and $\rho \in [2,m-1]$; the formula for $\sp(\Z_m)$
then follows from Lemma~\ref{l:lem23}.

Set $k:=\lfloor (m-2)/(\rho-1)\rfloor$, so that $k\ge 1$ in view of $\rho\le
m-1$. By Proposition~\ref{p:prop28}, we have $\tp(\Z_m)\le k+1$. Therefore, it
suffices to exhibit a $\rho$-maximal aperiodic subset $A\seq\Z_m$ of
cardinality $|A|\ge k+1$. Put $B:=\{0,1\longc k\}\seq\Z_m$, and let $A$ be a
$\rho$-maximal subset of $\Z_m$, containing $B$. It is enough to show that
$A$ is aperiodic; in fact this will imply that $A=B$.

Writing $H:=\pi(A)$, we obtain
$$
\<B\>_{\rho-1}^+ + H \seq \<A\>_{\rho-1}^+ + H = \<A\>_{\rho-1}^+ \snqq \Z_m.
$$
Now $\<B\>_{\rho-1}^+=\{0,1\longc(\rho-1)k\}$ and in particular
$$
  |\<B\>_{\rho-1}^+| = (\rho-1)k+1
                                > (\rho-1)\, \frac{(m-2)}{2(\rho-1)}+1 = m/2.
$$
Thus $\<B\>_{\rho-1}^+ + H\snqq\Z_m$ implies that $H=\{0\}$, as claimed.
\end{proof}


\section{$4$-maximal sets: proof of Theorem~\ref{t:thm27}}\label{s:bounds_III}

We begin with an easy consequence of Proposition~\ref{p:prop28}.
\begin{lemma}\label{l:lem61}
  Let $G$ be a finite abelian group, and let $n\in\N$. If
  $n\mid\exp(G)$ and $n\le |G|/3$, then $\tp[n+1](G)=|G|/n$.
\end{lemma}

\begin{proof}
By Proposition~\ref{p:prop28} it suffices to construct an aperiodic
$(n+1)$-maximal generating set $A$ for $G$ of cardinality $|A|=|G|/n$.

Clearly, there is a surjective homomorphism from $G$ onto the group $\Z_n$.
Let $H$ be the kernel of this homomorphism, so that $H$ is a proper subgroup
of $G$ with $|H|\ge 3$. Choose $g\in G$ such that $g+H$ generates $G/H$, and
put $A:=\{0\}\cup (g+H)\stm\{g\}$. Evidently, we have $|A|=|H|=|G|/n$, and
$|H|\ge 3$ implies that $\<A\>=G$. Furthermore, $\<A\>_n^+=G\stm\{g\}$ and
$A$ is maximal with this property. Finally, Lemma~\ref{lem51} shows that $A$
is aperiodic.
\end{proof}

Now let $G$ be a finite abelian group with $\diam^+(G) \ge 4$.
Proposition~\ref{p:prop28} and Lemma~\ref{l:lem23} show that
\begin{align}
  \tp[4](G) & \le \lfl \frac{|G| +1}{3} \rfl, \label{e:equ61} \\
  \sp[4](G) & \le \max \{ (|G|/d) \cdot \lfl (d+1)/3 \rfl \colon d
  \text{ divides } |G| \} \notag.
\end{align}
Our guiding idea is that in many situations these inequalities are
sharp.

\begin{proof}[Proof of Theorem~\ref{t:thm27}]
(i) If $3$ divides $|G|$, then Lemma~\ref{l:lem61} implies that
$\tp[4](G)=|G|/3$, unless $|G|=6$ in which case $G$ is cyclic and
$\tp[4](G)=2$ by Theorem~\ref{t:thm25}. Assume now that every divisor of $|G|$
is congruent to $1$ modulo $3$.

We argue by induction on $\rk(G)$. If $\rk(G)=1$, then $G$ is cyclic and
Theorem~\ref{t:thm25} yields $\tp[4](G)=(|G|-1)/3$.  Now suppose that
$\rk(G)\ge 2$. We write $G=G_1\oplus G_2$, where $G_1$ is cyclic of order
$m\equiv 1\mmod 3$ and $\rk(G_2)=\rk(G)-1$. To simplify the notation we
assume that $G_1=\Z_m$. Then $A_1:=\{0,1\longc (m-4)/3\}$ is an aperiodic
$4$-maximal subset of $G_1$ with $|A_1|=(|G_1|-1)/3$. By induction we find an
aperiodic $4$-maximal subset $A_2\seq G_2$ with $|A_2|=(|G_2|-1)/3$. We
define
$$
A := (A_1 + G_2) \cup (\{(m-1)/3\} + A_2) \seq G.
$$
Clearly, we have $|A|=|A_1||G_2|+|A_2|=(|G|-1)/3$ and hence $A$ is aperiodic.
Moreover, it is easily seen that
$G\stm\<A\>_3^+=\{m-1\}+(G_2\stm\<A_2\>_3^+)$ and $A$ is $4$-maximal. Thus
$\tp[4](G) \ge |A| = (|G| -1)/3$; the reverse inequality holds by
\eqref{e:equ61}.

\smallskip\noindent

(ii) Recall that by the assumption $G$ is not of exponent $2$, and that
$$
  \eta(G) := \max \{ |H| \colon H \le G \text{ with } \exp(G/H)
    \neq 2 \text{ and } |G/H| \equiv 2 \mmod{3} \}.
$$
Proposition~\ref{p:prop28} and the results \eqref{e:equ21} of Davydov-Tombak
show that for every $H \lneqq G$ we have
$$
|H| \cdot \tp[4](G/H) \leq
\begin{cases}
  \frac13\,|G|          & \text{if $|G/H| \equiv 0 \mmod{3}$,} \\
  \frac13\,(|G| - |H|)  & \text{if $|G/H| \equiv 1 \mmod{3}$,} \\
  \frac13\,(|G| + |H|)  & \text{if $|G/H| \equiv 2 \mmod{3}$ and $\exp(G/H) \neq 2$,} \\
  \frac14\,(|G| + 4|H|) & \text{if $|G/H| \equiv 2 \mmod{3}$, $\exp(G/H) = 2$,} \\
                        & \text{\hspace{1.65in} and $\rk(G/H) \geq 4$,} \\
  0                     & \text{if $G/H$ is isomorphic to a subgroup of $\Z_2^3$.}
\end{cases}
$$
Observe that in the fourth case, i.e.\ when $G/H \cong \Z_2^r$ with $r \ge
4$, we have $(|G| + 4|H|)/4 \le (|G| - |H|)/3$. Thus Lemma~\ref{l:lem23} yields
$$
\sp[4](G) \le
\begin{cases}
  \frac{1}{3} (|G| + \eta(G)) & \text{if $\eta(G) \neq 0$,} \\
  \frac{1}{3} |G| & \text{if $\eta(G) = 0$ and $3$ divides $|G|$,} \\
  \frac{1}{3} (|G| -1) & \text{if $\eta(G) = 0$ and $3$ does not
    divide $|G|$,}
\end{cases}
$$
and it remains to establish the reverse inequalities.

Assuming first that $\eta(G)\neq 0$, we find $H\le G$ such that
 $\exp(G/H)\neq 2,\ |G/H|\equiv 2\mmod{3}$, and $|H|=\eta(G)$. Considering the
prime factorization of $|G/H|$, we see that there are three possibilities:
\begin{itemize}
\item[(a)] $|G/H|=p$, where $p\equiv 2\mmod{3}$ is an odd prime, and then
  $G/H\cong\Z_p$;
\item[(b)] $|G/H|=2q$, where $q\equiv 1\mmod{3}$ is a prime, and then
  $G/H\cong\Z_{2q}$;
\item[(c)] $|G/H|=8$, and then $G/H\cong\Z_2\oplus\Z_4$ or $G/H\cong\Z_8$.
\end{itemize}
In cases (a) and (b) the group $G/H$ is cyclic of order $|G|/\eta(G)$, so
Lemma~\ref{l:lem23} and Theorem~\ref{t:thm25} show that
$$
  \sp[4](G) \ge |H| \cdot \tp[4](G/H) = \eta(G) \cdot (|G|/\eta(G) + 1)/3
              = \tfrac{1}{3} (|G| + \eta(G)).
$$
In case (c) the claim follows similarly, noting that
$\tp[4](\Z_2\oplus\Z_4)=\tp[4](\Z_8)=3$, by Theorems~\ref{t:thm24} and
\ref{t:thm25} respectively.

Next, if $\eta(G) = 0$ and $3$ divides $|G|$, then by~(i) we have
$\sp[4](G)\ge\tp[4](G)=|G|/3$.

Finally, suppose that $\eta(G) = 0$ and that $3$ does not divide
$|G|$. Then, clearly, $|G|$ has no odd prime divisors congruent to $2$
modulo $3$, and $|G|$ is not a power of $2$. (If $G$ were a $2$-group,
then, because $\diam^+(G) \ge 4$ and $G$ is not of exponent $2$, we
could map $G$ onto $\Z_8$ or $\Z_4 \oplus \Z_2$ in contradiction to
$\eta(G) = 0$.) Consequently, $|G|$ has at least one prime divisor
congruent to $1$ modulo $3$, and so $2$ does not divide $|G|$ at all.
Thus every (prime) divisor of $|G|$ is congruent to $1$ modulo $3$,
and the claim follows again from~(i).
\end{proof}


\section{General Estimates for $\tp(G)$ and $\sp(G)$: \\ Proofs of
  Theorems~\ref{t:thm29} and~\ref{t:thm210}}\label{s:bounds_IV}

\begin{proof}[Proof of Theorem~\ref{t:thm29}]
  It suffices to show that $\sp(G)\le\frac2{\rho+1}\,|G|$ and that (i)
  implies (ii). Indeed, (ii) trivially implies (i), and it is easily
  seen that, if (i) and (ii) are equivalent, then equality holds in
  $\sp(G)\le\frac2{\rho+1}\,|G|$ if and only if $G$ has a subgroup $H$
  such that $G/H$ is cyclic of order $\rho+1$.

Let $A$ be a generating set for $G$ with $\diam_A^+(G)\ge\rho$.  We write
$H:=\pi(\<A\>_{\rho-1}^+)$ and note that $|G|\ge|\<A\>_{\rho-1}^+|+|H|$.
Since
  $$ \<A\>_{\rho-1}^+ = (\rho-1)A_0+H = (\rho-1)(A_0+H), $$
Kneser's Theorem yields
  $$ |\<A\>_{\rho-1}^+| \ge (\rho-1)|A_0+H| - (\rho-2)|H|. $$
Combining these observations, we obtain
\begin{equation}\label{e:equ71}
   |G| \ge (\rho-1)|A_0+H| - (\rho-3)|H|.
\end{equation}

Since $0\in A_0\nsubseteq H$, the set $A_0+H$ is the union of at least two
$H$-cosets, and so $|A_0+H|\ge 2|H|$. From \eqref{e:equ71} we obtain
\begin{equation}\label{e:equ72}
  |G| \ge \Big(\rho-1-\frac12\, (\rho-3)\Big) |A_0+H| \ge
   \frac{\rho+1}2\,|A|.
\end{equation}
This gives the upper bound $\sp(G)\le\frac2{\rho+1}\,|G|$; it remains to show
that (i) implies (ii).

To this end, in addition to our earlier assumption that $A$ is a generating
set for $G$ with $\diam_A^+(G)\ge\rho$, we assume that
$|A|=\frac2{\rho+1}\,|G|$. Then \eqref{e:equ72} yields $A=A_0+H$, and this
set is the union of exactly two $H$-cosets. Consequently, $A=H\cup(g+H)$ with
some $g\in G\stm H$. Since $A$ generates $G$, the quotient group $G/H$ is
cyclic and generated by $g+H$. Since
  $$ \textstyle \bigcup\, \{ kg+H \colon 0\le k\le \rho-1 \}
                                             = \<A\>_{\rho-1}^+\snqq G, $$
the order of $G/H$ is at least $\rho+1$. Finally, in view of
$|A|=\frac2{\rho+1}\,|G|$, the estimate
  $$ |A| = 2|H| = 2\,\frac{|G|}{|G/H|} \le 2\,\frac{|G|}{\rho+1} = |A| $$
yields $|G/H|=\rho+1$, as desired.
\end{proof}

\begin{proof}[Proof of Theorem~\ref{t:thm210}]
Let $A\seq G$ be a generating set such that $\diam_A^+(G)\ge\rho$ and
$|A|>\frac32\,\rho^{-1}|G|$. We set $H:=\pi(\<A\>_{\rho-1}^+)$ and show that
there exists $g\in G$ such that
\begin{itemize}
\item[(i)] $A\seq H\cup(g+H)$, and $|A|>\big(3-\frac{\rho-3}{3\rho-1}\big) |H|$;
\item[(ii)] $G/H$ is cyclic of order $\rho+1\le |G/H|\le \frac43\,\rho$, and
  $g+H$ generates $G/H$.
\end{itemize}

As in the proof of Theorem~\ref{t:thm29}, we see that $A_0+H$ consists of
two $H$-cosets: the possibility $|A_0+H|\ge 3|H|$ is ruled out, since
otherwise
  $$ |G| \ge \Big( \rho-1-\frac13\,(\rho-3) \Big) |A_0+H|
                                                 \ge \frac 23\, \rho |A|, $$
contrary to the assumptions. Exactly as before, the quotient group $G/H$ is
cyclic of order at least $\rho+1$, and there exists $g\in G$ such that $g+H$
is a generator of $G/H$ and $A_0+H=H\cup (g+H)$.

Finally, from
  $$ 2|H| = |A_0+H| \ge |A| > \frac32\,\rho^{-1}|H||G/H| $$
it follows that $|G/H|<\frac43\,\rho$ and
  $$ |A| > \frac{3(\rho+1)}{2\rho}\,|H|
                                = \left(2-\frac{\rho-3}{2\rho}\right) |H|, $$
as required.
\end{proof}

In fact one can go further and, given $\eps>0$, describe in a similar way the
structure of those generating sets $A$ for $G$ which satisfy
$\diam_A^+(G)\ge\rho$ and $|A|>(1+\eps)|G|/(\rho-1)$. Indeed, in this case
$A_0+H$ consists of $O(\eps^{-1})$ $H$-cosets. However, for small $\eps$ the
description is likely to become quite complicated.


\appendix

\section{An Example: $\tp(G)$ May Vanish}

As can be seen directly from the definition, for every abelian group $G$ and
every $\rho\in[1,\diam^+(G)]$ we have $\sp(G)>0$. Theorem~\ref{t:thm24}, on
the other hand, contains the observation that $\tp[1](G)=0$ for every finite
abelian group $G$. In this appendix we construct examples, illustrating that
$\tp(G)$ may vanish for large values of $\rho$ for less trivial reasons.

\begin{example}
For $n\in \N$ with $n \ge 3$, let $G = \Z_2 \oplus \Z_{2^n}$. Then
$\diam^+(G) = 2^n$, and every $(2^n -2)$-maximal subset of $G$ is periodic.
\end{example}

\begin{proof}[Explanation]
  Let $G = \Z_2 \oplus \Z_{2^n}$ with $n \ge 3$.
  Theorem~\ref{t:thm21} shows that $\diam^+(G) = 2^n$. Let $A \seq G$ be
  maximal subject to $\<A\>_{2^n -2}^+ \neq \<A\> = G$; we have to
  show that $A$ is periodic.

  By maximality, we have $A = A_0$. Since $A$ generates $G$, there
  exists an element $a_1 \in A$ of order $\ord(a_1) = 2^n$, and
  furthermore there exists an element $a_2 \in A$ such that $a_2
  \notin \<a_1\>$. Without loss of generality, we may assume that $a_1
  = (0,1)$ and $a_2 = (1,\alp)$ with $0 \le \alp < 2^n$.

  Notice that, whenever $(1,\bet) \in A$ (this holds, for instance, for
  $\bet = \alp$), then
  \begin{equation}\label{e:equA1}
    G \stm \<A\>_{2^n -2}^+ \seq \{(0,2^n -1), (1,\bet + 2^n -2),
    (1,\beta +2^n-1)\},
  \end{equation}
  and because $\<A\>_{2^n -2}^+ \neq G$, at least one of $(0,2^n -1)$,
  $(1,\beta + 2^n- 2)$, $(1,\beta + 2^n -1)$ has length greater than
  $2^n -2$ with respect to $A$.

  \noindent
  \emph{Assertion. If $b = (1,\bet)\in A$, then
    $b\in\{(1,0),(1,1),(1,2^{n-1}),(1,2^{n-1}+1)\}$.} \\
  For a contradiction, suppose that $b = (1,\bet) \in
  A\stm\{(1,0),(1,1),(1,2^{n-1}),(1,2^{n-1}+1)\}$. Then we have $2b =
  (0,2 \bet) \notin \{(0,0),(0,1),(0,2),(0,3),(0,2^n -1)\}$, and thus
  \begin{align*}
    & (0,2^n -1) - 2b = (0,2^n -1 -2\bet) \in \<a_1\>_{2^n-5}^+, & &
    \text{so $l_A(0,2^n -1) \le 2^n -3$;}\\
    &(1,\beta + 2^n - 2) - 3b = (0,2^n -2- 2\bet) \in
    \<a_1\>_{2^n-6}^+, & & \text{so $l_A(1,\beta + 2^n -2) \le 2^n - 3$;}\\
    &(1,\beta + 2^n - 1) - 3b = (0,2^n -1 -2\bet) \in
    \<a_1\>_{2^n-5}^+, & & \text{so $l_A(1,\beta + 2^n -1) \le 2^n -
    2$.}
  \end{align*}
  This contradicts the observation following \eqref{e:equA1}.

  \noindent
  \emph{Assertion. If $c=(0,\gamma)\in A$, then $c\in\{(0,0),(0,1)\}$.} \\
  Let $c = (0,\gamma) \in G \stm \{(0,0),(0,1)\}$. Then it is easily
  seen that $\< a_1, c \>_{2^n-3}^+ = \{ (0,\delta) \colon \delta \in
  \Z_{2^n} \}$, and hence $\< a_1, a_2, c \>_{2^n -2}^+ = G$. It
  follows that $c \not \in A$, as required.

  The two assertions above yield
  $$
  \{ (0,0),(0,1), (1,\alp) \} \seq A \seq \{ (0,0), (0,1), (1,0),
  (1,1), (1,2^{n-1}), (1,2^{n-1}+1) \}.
  $$

  \noindent
  \emph{Case 1: $(1,\alp) = (1,0)$.} Because of \eqref{e:equA1} we
  certainly have $(1,2^{n-1}), (1,2^{n-1} +1) \notin A$. On the other
  hand, $(0,2^n -1) \notin \< (0,1),(1,0),(1,1) \>_{2^n - 2}^+$. So $A
  = \{ (0,0), (0,1), (1,0), (1,1) \}$ has non-trivial period $\{
  (0,0), (1,0) \}$.

  \noindent
  \emph{Case 2: $(1,\alp) = (1,1)$.} Again \eqref{e:equA1} shows that
  $(1,2^{n-1}), (1,2^{n-1} +1) \notin A$. So as in in the first case
  $A = \{ (0,0), (0,1), (1,0), (1,1) \}$ has non-trivial period $\{
  (0,0), (1,0) \}$.

  \noindent
  \emph{Case 3: $(1,\alp) \in \{ (1,2^{n-1}), (1,2^{n-1}+1) \}$.}
  Write the elements of $G$ with respect to the generating pair
  $((1,2^{n-1}),(0,1))$ rather than $((1,0),(0,1))$. In these new
  coordinates, $a_1$ is still represented by $(0,1)$, but $a_2$ is
  represented by $(1,0)$ or $(1,1)$. We are reduced to either Case~1
  or Case~2.
\end{proof}

Perhaps it is feasible to classify all pairs $(G,\rho)$ such that $\tp(G)=0$,
but beyond the above example nothing is known.


\end{document}